\documentclass[reqno,12pt]{amsart}
\usepackage[english]{babel}
\usepackage{amsmath,amsfonts,amsthm,amssymb,amsbsy,upref,color,graphicx,hyperref,enumerate,comment,tikz}
\usepackage[a4paper,margin=3.00truecm]{geometry}
\usepackage{bbm}
\usepackage{soul}
\usepackage[hyperpageref]{backref}

\def\N{\mathbb{N}}
\def\R{\mathbb{R}}

\def\HH{\mathcal{H}}
\def\M{\mathcal{M}}
\def\O{\Omega}

\def\eps{\varepsilon}
\def\ds{\displaystyle}

\newcommand{\dx}{\,dx}
\newcommand{\dH}{\,d\mathcal{H}^{d-1}}
\newcommand{\be}{\begin{equation}}
\newcommand{\ee}{\end{equation}}

\newcommand{\bib}[4]{\bibitem{#1}{\sc#2: }{\it#3. }{#4.}}

\numberwithin{equation}{section}
\theoremstyle{plain}
\newcommand{\mres}{\mathbin{\vrule height 1.6ex depth 0pt width
0.13ex\vrule height 0.13ex depth 0pt width 1.3ex}}

\newtheorem{theo}{Theorem}[section]
\newtheorem{lemm}[theo]{Lemma}

\newtheorem{prop}[theo]{Proposition}

\theoremstyle{definition}
\newtheorem{rema}[theo]{Remark}

\title[Eigenvalue and torsion with Robin conditions]{Relations between principal eigenvalue and torsional rigidity with Robin boundary conditions}

\author[G. Buttazzo]{Giuseppe Buttazzo}
\author[S. Cito]{Simone Cito}
\author[F. Solombrino]{Francesco Solombrino}

\date{\today}

\begin{document}

\maketitle


\begin{abstract}
We consider the torsional rigidity and the principal eigenvalue related to the Laplace
operator with Dirichlet and Robin boundary conditions. The goal is to find upper and lower bounds to products of suitable powers of the quantities above in the class of Lipschitz domains. The threshold exponent for the Robin case is explicitly recovered and shown to be strictly smaller than in the Dirichlet one.
\end{abstract}



\section{Introduction}\label{sintro}

Elliptic boundary value problems involving Robin boundary conditions share many structural and analytical similarities with those formulated under Dirichlet boundary conditions. Nevertheless, the Robin case also exhibits distinct and subtle phenomena that are of considerable mathematical and physical interest. In the present article, we focus on exploring the relations between two fundamental quantities associated with the Laplace operator under Robin conditions: the principal (or first) eigenvalue and the torsional rigidity of a domain.

Throughout the paper, we denote by $\O$ an open bounded subset of $\R^d$ with a sufficiently regular boundary (for instance of Lipschitz class). Its Lebesgue measure is denoted by $|\O|$, and without loss of generality we shall assume that $|\O|=1$ in order to simplify the presentation, 

Let $\beta>0$ be the Robin parameter, the {\it first eigenvalue} $\lambda_\beta(\O)$ of the Laplacian with Robin boundary conditions is defined as the smallest $\lambda$ for which the boundary value problem
$$\begin{cases}
-\Delta u=\lambda u&\text{in }\O\\
\partial_\nu u+\beta u=0&\text{on }\partial\O
\end{cases}$$
admits a nontrivial solution, where $\partial_\nu$ denotes the outward normal derivative on $\partial\O$. Equivalently, $\lambda_\beta(\O)$ can be characterized variationally as the minimum of the associated Rayleigh quotient
\be\label{reigen}
\lambda_\beta(\O)=\min\Bigg\{\frac{\int_\O|\nabla u|^2\,dx+\beta\int_{\partial\O}u^2\,d\HH^{d-1}}{\int_\O u^2\,dx}\ :\ u\in H^1(\O)\setminus\{0\}\Bigg\},
\ee
where $\HH^{d-1}$ denotes the Hausdorff $d-1$ dimensional measure.

The second quantity of interest is the {\it torsional rigidity} $T_\beta(\O)$ which, in the context of elasticity theory, represents the resistance of the domain $\O$ to torsion under boundary conditions of Robin type. It is defined by $$T_\beta(\O)=\int_\O w_\beta\,dx,$$
where $w_\beta$ denotes the unique weak solution of the Poisson problem
$$\begin{cases}
-\Delta u=1&\text{in }\O\\
\partial_\nu u+\beta u=0&\text{on }\partial\O.
\end{cases}$$
We refer to \cite{VDBB14} for an extended presentation of problems involving the Robin torsional rigidity. Equivalently, $T_\beta(\O)$ admits the variational representation
\be\label{rtorsi}
T_\beta(\O)=\max\Bigg\{\frac{\big(\int_\O u\,dx\big)^2}{\int_\O|\nabla u|^2\,dx+\beta\int_{\partial\O}u^2\,d\HH^{d-1}}\ :\ u\in H^1(\O)\setminus\{0\}\Bigg\}.
\ee

When $\beta=+\infty$ the Robin condition degenerates into the Dirichlet condition
$u=0$ on $\partial\O$. In this limit, the corresponding quantities $\lambda_\infty(\O)$ and $T_\infty(\O)$ reduce to the classical Dirichlet principal eigenvalue and Dirichlet torsional rigidity respectively, and the minimization problems \eqref{reigen} and \eqref{rtorsi} are naturally restricted to the Sobolev space $H^1_0(\O)$. It is well known (see for instance \cite[Proposition 4.5]{BFK} for the eigenvalue and \cite[Theorem 3.1]{BO25} with $f\equiv1$ for the torsional rigidity) that
\be\label{infty}
\lim_{\beta\to+\infty}\lambda_\beta(\O)=\lambda_\infty(\O)\qquad\text{and}\qquad\lim_{\beta\to+\infty}T_\beta(\O)=T_\infty(\O).
\ee

Our main goal is to investigate shape optimization problems involving the interaction between these two fundamental quantities. For an overview on shape optimization problems we refer to the books \cite{BB05}, \cite{he06}, \cite{HP18}. Specifically, for a fixed $\beta>0$ and a given real exponent $q>0$ we consider shape functionals of the form
\be\label{fbq}
F_{\beta,q}(\O)=\lambda_\beta(\O)T^q_\beta(\O).
\ee
We are particularly interested in determining the extremal values of the functional $F_{\beta,q}$ among all domains $\O$ of fixed measure: 
\be\label{mMbq}
\begin{cases}
m_{\beta,q}=\inf\big\{F_{\beta,q}(\O)\ :\ |\O|=1\big\}\\
M_{\beta,q}=\sup\big\{F_{\beta,q}(\O)\ :\ |\O|=1\big\}.
\end{cases}\ee
In the following, since the Robin parameter $\beta$ is fixed, we often omit it from the notation and simply write $F_q$, $m_q$, $M_q$.

The Dirichlet case corresponding to $\beta=+\infty$ has been extensively analyzed in \cite{BBG24} and \cite{BBP22} (see also \cite{VDBBP21}, \cite{VDBBV15}). In the present work, we turn our attention to the more delicate situation where $\beta<+\infty$, and we highlight an interesting phenomenon that emerges in this regime. Indeed, while it is known that $m_{\infty,q}=0$ if and only if $q>2/(d+2)$, we establish that $m_{\beta,q}=0$ if and only if $q>1/(d+1)$, independently of the value of $\beta$, in spite of the asymptotic behavior described in \eqref{infty}.

As for the Dirichlet boundary condition, the case $q=1$ deserves particular attention. Indeed, we prove that $M_1=1$ by carrying out a detailed analysis of the asymptotic behavior of solutions in appropriately and finely perforated domains. This result is similar to what is known in the Dirichlet setting. However, in the Robin case the characteristic size of the perforations is different, which leads to a different scaling, even though the final conclusion remains the same.

The paper is organized as follows. In Section \ref{spreli}, we recall some preliminary analytical results and technical tools that will be employed in the subsequent analysis. Section \ref{sdiric} is devoted to the Dirichlet case $\beta=+\infty$, for which several sharp results are already known; we summarize and reinterpret them in a unified framework. In Section \ref{srobin}, we turn to the Robin case $\beta<+\infty$, which presents new challenges and a richer phenomenology. Finally, in Section \ref{sopenq} we conclude by outlining a number of open problems and conjectures that, in our view, deserve further investigation.

\section{Preliminary results}\label{spreli}

If we consider $\lambda_\beta$ or $T_\beta$ only, the related optimization problems are well-known and we have the following results. Here and in the following we denote by $B$ a ball in $\R^d$ with unitary Lebesgue measure.

\begin{itemize}
\item We have
$$\lambda_\infty(B)\le\lambda_\infty(\O)\qquad\text{for every $\O$ with }|\O|=1.$$
This inequality is known as {\it Faber-Krahn inequality}.

\item We have
$$T_\infty(B)\ge T_\infty(\O)\qquad\text{for every $\O$ with }|\O|=1.$$
This inequality is known as {\it Saint-Venant inequality}.

\item We have for every $\beta<+\infty$
$$\lambda_\beta(B)\le\lambda_\beta(\O)\qquad\text{for every $\O$ with }|\O|=1.$$
This inequality was proved in \cite{Bo86} in the case $d=2$ and in \cite{Da06} for any dimension $d$. In \cite{BG15} the same result is obtained by means of a free discontinuity approach in spaces of functions of bounded variation.

\item We have for every $\beta<+\infty$
$$T_\beta(B)\ge T_\beta(\O)\qquad\text{for every $\O$ with }|\O|=1.$$
This inequality was proved in \cite{BG15}.

\end{itemize}

The main difference between the Dirichlet case $\beta=+\infty$ and the Robin case $\beta<+\infty$ is in the scaling properties. We have for every scaling factor $t>0$
$$\begin{cases}
\lambda_\infty(t\O)=t^{-2}\lambda_\infty(\O)\\
T_\infty(t\O)=t^{d+2}T_\infty(\O),
\end{cases}$$
while, if $\beta<+\infty$ we have
\begin{equation}\label{eq: scaling}
\begin{cases}
\lambda_\beta(t\O)=t^{-2}\lambda_{t\beta}(\O)\\
T_\beta(t\O)=t^{d+2}T_{t\beta}(\O),
\end{cases}
\end{equation}
An easy comparison, by using either $H^1_0(\O)$ or constant test functions, gives for every $\O$
\begin{equation*}
\begin{cases}
\ds\lambda_\beta(\O)\le\lambda_\infty(\O)\wedge\beta\frac{\mathcal{H}^{d-1}(\partial\O)}{|\O|}\\
\ds T_\beta(\O)\ge T_\infty(\O)\vee\frac{|\O|^2}{\beta\mathcal{H}^{d-1}(\partial\O)}.
\end{cases}
\end{equation*}
In particular, for a ball $B_r$ of radius $r$ we have
$$\begin{cases}
\lambda_\beta(B_r)\le r^{-2}\lambda_\infty(B_1)\wedge r^{-1}\beta d\\
\ds T_\beta(B_r)\ge r^{d+2}T_\infty(B_1)\vee r^{d+1}\frac{\omega_d}{d},
\end{cases}$$
where $\omega_d$ denotes the Lebesgue measure of the unit ball in $\R^d$. Some more precise estimates hold.

\begin{prop}\label{estim}
For every $r>0$ we have:
\be\label{lbr}
\frac{\beta}{4r(1+\beta r)}\le\lambda_\beta(B_r)\le\frac{C_d\beta}{r(1+\beta r)}
\ee
where $C_d$ is a constant depending only on the dimension $d$, and
\be\label{tbr}
T_\beta(B_r)=\frac{\omega_d}{d}\Big(\frac{r^{d+1}}{\beta}+\frac{r^{d+2}}{d+2}\Big).
\ee
\end{prop}

\begin{proof}
Estimates \eqref{lbr} can be found in \cite{Ko14}, while equality \eqref{tbr} comes by a direct computation in polar coordinates.
\end{proof}

In the following we need a variant of Gagliardo-Nirenberg inequality for functions in $H^1(\O)$; we recall the inequality below, as obtained in \cite{W02}.

\begin{theo}\label{warma}
For every open bounded subset $\O$ of $\R^d$ with a Lipschitz boundary there exists a constant $C>0$ such that
$$\int_\O u^2\,dx\le C\bigg[\int_\O|\nabla u|^2\,dx+\int_{\partial\O}u^2\,d\HH^{d-1}\bigg]^{d/(d+1)}\bigg[\int_\O|u|\,dx\bigg]^{2/(d+1)}$$
for all $u\in H^1(\O)\cap L^1(\O)$. Moreover, the constant $C$ depends only on $d$ and on $|\O|$.
\end{theo}

\begin{proof} See Proposition 4.2.9 of \cite{W02}.
\end{proof}

\section{The case of Dirichlet boundary conditions}\label{sdiric}

The optimization problems for the shape functionals $F_q$ in \eqref{fbq} have been extensively studied in the Dirichlet case $\beta=+\infty$. We recall in the present section the most important results and we refer to \cite{BDM93} for a general result on shape optimization problems with Dirichlet boundary conditions. We start by the minimization problem in \eqref{mMbq}.

\begin{theo}\label{mqinfty}
In the case $\beta=+\infty$ the following facts hold for the quantity $m_q$ in \eqref{mMbq}:
\begin{itemize}
\item[-]for every $q>2/(d+2)$ we have $m_q=0$;
\item[-]for every $q\le2/(d+2)$ we have $m_q>0$; in addition, the value $m_q$ is reached by the ball $B$.
\end{itemize}
\end{theo}

\begin{proof}
The first assertion can be proved by taking as $\O$ the disjoint union of a ball $B_\delta$ and of $N$ balls $B_\eps$, with $\delta>\eps$ and $\omega_d(\delta^d+N\eps^d)=1$. We have
\[\begin{split}
F_q(\O)&=\delta^{-2}\lambda_\infty(B_1)\Big(\big(\delta^{d+2}+N\eps^{d+2}\big)T_\infty(B_1)\Big)^q\\
&=F_q(B_1)\delta^{-2}\big(\delta^{d+2}+\eps^2(1/\omega_d-\delta^d)\big)^q.
\end{split}\]
By taking $\eps^2\ll\delta^{d+2}$ we obtain
$$F_q(\O)\approx F_q(B_1)\delta^{-2+q(d+2)},$$
and the first assertion is proved by passing to the limit as $\delta\to0$.

For the second assertion, the equality $m_q=F_q(B)$ when $q\le2/(d+2)$ was proved in \cite{KJ78a} and \cite{KJ78b}; for an extension to some nonlinear cases we refer to \cite{Br14}. We give here a simple proof of the inequality $m_q>0$, that can be obtained by the Gagliardo-Nirenberg inequality
\be\label{gn}
\|u\|_{L^2(\O)}\le C\|\nabla u\|^\theta_{L^2(\O)}\|u\|^{1-\theta}_{L^1(\O)}\qquad\text{for every }u\in H^1_0(\O),
\ee
where $\theta=d/(d+2)$ and $C$ is a constant that depends only on the dimension $d$.

In fact, if $u$ is a first eigenfunction of $-\Delta$ with Dirichlet boundary conditions, by the definition \eqref{rtorsi} of $T_\infty$ we obtain
\[\begin{split}
F_q(\O)&\ge\frac{\int_\O|\nabla u|^2\,dx}{\int_\O|u|^2\,dx}\bigg(\frac{\big(\int_\O u\,dx\big)^2}{\int_\O|\nabla u|^2\,dx}\bigg)^q\\
&=\bigg(\frac{\|\nabla u\|^{1-q}_{L^2(\O)}\|u\|^q_{L^1(\O)}}{\|u\|_{L^2(\O)}}\bigg)^2.
\end{split}\]
By the Gagliardo-Nirenberg inequality \eqref{gn} we have
\[\begin{split}
F_q(\O)&\ge\frac{1}{C^2}\bigg(\frac{\|\nabla u\|_{L^2(\O)}}{\|u\|_{L^1(\O)}}\bigg)^{1-q-\theta}\\
&=\frac{1}{C^2}\big(\lambda_\infty(\O)\big)^{(1-q-\theta)/2}\bigg(\frac{\|u\|_{L^2(\O)}}{\|u\|_{L^1(\O)}}\bigg)^{1-q-\theta}.
\end{split}\]
Now, the Faber-Krahn inequality, the fact that $1-q-\theta\ge0$, and the H\"older inequality, give
$$F_q(\O)\ge\frac{1}{C^2}\big(\lambda_\infty(B)\big)^{(1-q-\theta)/2}.$$
Hence, taking the infimum over all $\O$ with $|\O|=1$, gives
$$m_q\ge\frac{1}{C^2}\big(\lambda_\infty(B)\big)^{(1-q-\theta)/2},$$
as required.
\end{proof}

We now pass to the quantity $M_q$ in \eqref{mMbq}.

\begin{theo}\label{Mqinfty}
In the case $\beta=+\infty$ the following facts hold for the quantity $M_q$ in \eqref{mMbq}:
\begin{itemize}
\item[-]for every $q<1$ we have $M_q=+\infty$;
\item[-]for $q=1$ we have $M_q=1$;
\item[-]for every $q>1$ we have $M_q<+\infty$;
\item[-]there exists $\bar q>1$ such that for every $q\ge\bar q$ the value $M_q$ is reached by the ball $B$.
\end{itemize}
\end{theo}

\begin{proof}
The first assertion follows by taking as $\O$ the disjoint union on $N$ balls $B_\eps$, with $\omega_d N\eps^d=1$. This gives
$$F_q(\O)=\eps^{-2}\lambda_\infty(B_1)\big(N\eps^{d+2}T_\infty(B_1)\big)^q=F_q(B_1)\eps^{-2+2q},$$
so that, passing to the limit as $\eps\to0$, we have $M_q=+\infty$ when $q<1$.

In the second assertion, to prove that $M_1\le1$ it is enough to consider the solution $w$ of the PDE
$$\begin{cases}
-\Delta u=1&\text{in }\O\\
u=0&\text{on }\partial\O,
\end{cases}$$
and take it as a test function for $\lambda_\infty$. This gives
$$F_1(\O)\le\frac{\int_\O|\nabla w|^2\,dx}{\int_\O w^2\,dx}\frac{\big(\int_\O w\,dx\big)^2}{\int_\O|\nabla w|^2\,dx}=\frac{\big(\int_\O w\,dx\big)^2}{\int_\O|w|^2\,dx}$$
and, by H\"older inequality, we conclude that $M_1\le1$. The proof that $M_1=1$ is more delicate and we refer to \cite{VDBFNT16}; a simpler proof, that uses capacitary measures (for which we refer to \cite{BB05}), can be found in \cite{BBP22}.

To prove the third assertion it is enough to write
$$F_q(\O)=F_1(\O)\big(T_\infty(\O)\big)^{q-1}.$$
By the inequality $F_1(\O)\le1$ and the Saint-Venant inequality $T_\infty(\O)\le T_\infty(B)$, we obtain
$$F_q(\O)\le\big(T_\infty(B)\big)^{q-1},$$
so that $M_q\le\big(T_\infty(B)\big)^{q-1}$.

For the last assertion we refer to \cite{BBG24}, where this property was conjectured under the name of {\it reverse Kohler-Jobin inequality}, and to \cite{BLNP23}, where it was proved.
\end{proof}

\section{The case of Robin boundary conditions}\label{srobin}

In this section we consider the Robin case $\beta<+\infty$. We start by studying the case of $m_q$.

\begin{theo}\label{mqb}
In the case $\beta<+\infty$ the following facts hold for the quantity $m_q$ in \eqref{mMbq}:
\begin{itemize}
\item[-]for every $q>1/(d+1)$ we have $m_q=0$;
\item[-]for every $q\le1/(d+1)$ we have $m_q>0$.
\end{itemize}
\end{theo}

\begin{proof}
The first assertion can be proved by taking as $\O$ the disjoint union of a ball $B_\delta$ and of $N$ balls $B_\eps$, with $\delta>\eps$ and $\omega_d(\delta^d+N\eps^d)=1$. We have by \eqref{lbr} and \eqref{tbr}
\[\begin{split}
F_q(\O)&\approx\delta^{-1}\Big(\big(\delta^{d+1}+N\eps^{d+1}\big)\Big)^q\\
&=\delta^{-1}\big(\delta^{d+1}+\eps(1/\omega_d-\delta^d)\big)^q.
\end{split}\]
By taking $\eps\ll\delta^{d+1}$ we obtain
$$F_q(\O)\approx\delta^{-1+q(d+1)},$$
and the first assertion is proved by passing to the limit as $\delta\to0$.

For the second assertion, by the definition of the torsional rigidity $T_\beta$, denoting by $u$ the first eigenfunction of $-\Delta$ with Robin boundary condition on $\partial\O$, we have
\[\begin{split}
F_{\beta,q}(\O)&\ge\frac{\int_\O|\nabla u|^2\,dx+\beta\int_{\partial\O}u^2\,d\HH^{d-1}}{\int_\O u^2\,dx}\bigg[\frac{\big(\int_\O u\,dx\big)^2}{\int_\O|\nabla u|^2\,dx+\beta\int_{\partial\O}u^2\,d\HH^{d-1}}\bigg]^q\\
&=\frac{\big[\int_\O|\nabla u|^2\,dx+\beta\int_{\partial\O}u^2\,d\HH^{d-1}\big]^{1-q}\big[\int_\O u\,dx\big]^{2q}}{\int_\O u^2\,dx}.
\end{split}\]
If $q=1/(d+1)$ the assertion is proved as soon as
$$\int_\O u^2\,dx\le C\bigg[\int_\O|\nabla u|^2\,dx+\int_{\partial\O}u^2\,d\HH^{d-1}\bigg]^{d/(d+1)}\bigg[\int_\O u\,dx\bigg]^{2/(d+1)}$$
with a constant $C>0$ which does not depend on $\O$ when $|\O|=1$, which follows from Theorem \ref{warma}.

If $q<1/(d+1)$ we have, by the optimality of $B$ for $T_\beta$,
$$F_{\beta,q}(\O)=F_{\beta,1/(d+1)}(\O)\big[T_\beta(\O)\big]^{q-1/(d+1)}\ge F_{\beta,1/(d+1)}(\O)\big[T_\beta(B)\big]^{q-1/(d+1)},$$
and the assertion now follows from the case $q=1/(d+1)$.
\end{proof}

\begin{rema}
Notice the difference between the threshold $2/(d+2)$ of the case $\beta=+\infty$ and $1/(d+1)$ of the case $\beta<+\infty$. Then, for $1/(d+1)<q\le2/(d+2)$ we have $m_{\beta,q}=0$ for every $\beta<+\infty$, while $m_{\infty,q}>0$. This in spite of the fact that $\lambda_\beta(\O)$ and $T_\beta(\O)$ tend to $\lambda_\infty(\O)$ and $T_\infty(\O)$ respectively, as $\beta\to+\infty$, and so
$$\lim_{\beta\to+\infty}F_{\beta,q}(\O)=F_{\infty,q}(\O)\qquad\text{for every }\O.$$
\end{rema}

We now pass to consider the case of $M_q$. In the following theorem, we prove the finiteness of $M_q$ for $q>1$, and the equality $M_1=1$ that we obtain by means of homogenization. To do so, we construct a suitable sequence of smooth domains with a periodic lattice of smooth holes. We inform the reader that some parts of the following proof, namely the estimate on the eigenvalues  \eqref{eq:homoconvlambda}, could also be deduced within the abstract framework devised in \cite[Theorem 1]{K85} (which can be seen as the Robin counterpart of the classical results in \cite{ciomur}); nevertheless, we give an independent complete proof for the sake of self-containment.

\begin{theo}\label{Mqb}
In the case $\beta<+\infty$ the following facts hold for the quantity $M_q$ in \eqref{mMbq}:
\begin{itemize}
\item[-]for every $q<1$ we have $M_q=+\infty$;
\item[-]for $q=1$ we have $M_q=1$;
\item[-]for every $q>1$ we have $M_q<+\infty$;
\end{itemize}
\end{theo}

\begin{proof}
The first assertion follows by taking as $\O$ the disjoint union on $N$ balls $B_\eps$, with $\omega_d N\eps^d=1$. This gives by \eqref{lbr} and \eqref{tbr}
$$F_q(\O)\approx\eps^{-1}\big(N\eps^{d+1}\big)^q\approx\eps^{-2+2q},$$
so that, passing to the limit as $\eps\to0$, we have $M_q=+\infty$ when $q<1$.

In the second assertion, to prove that $M_1\le1$ it is enough to consider the solution $w$ of the PDE
$$\begin{cases}
-\Delta u=1&\text{in }\O\\
\partial_\nu u+\beta u=0&\text{on }\partial\O,
\end{cases}$$
and take it as a test function for $\lambda_\beta$. This gives
$$F_1(\O)\le\frac{\int_\O|\nabla w|^2\,dx+\beta\int_{\partial\O}w^2\,d\HH^{d-1}}{\int_\O w^2\,dx}\frac{\big(\int_\O w\,dx\big)^2}{\int_\O|\nabla w|^2\,dx+\beta\int_{\partial\O}w^2\,d\HH^{d-1}}=\frac{\big(\int_\O w\,dx\big)^2}{\int_\O|w|^2\,dx}$$
and, by H\"older inequality, we conclude that $M_1\le1$.

In order to prove the third assertion it is enough to write
$$F_q(\O)=F_1(\O)T^{q-1}_\beta(\O)\le M_1 T^{q-1}_\beta(\O)$$
and to use the inequality $M_1\le1$ and $T_\beta(\O)\le T_\beta(B)$ seen in Section \ref{spreli}.

It remains to prove that $M_1\ge1$, for which we have to construct a sequence $\O_n$ such that $\lim_n F_1(\O_n)\ge1$. 

We start considering a unit cube $Q_1=(0,1)^d$ perforated by small balls arranged on a periodic lattice.  
For $N \in \N$  and $\vec{\ell} \in \{1,...N\}^d$, let $x_{\vec{\ell}}$ denote the lattice points of spacing $1/N$, and remove balls of radius 
\[
r_N = \frac{k}{N^{d/(d-1)}}, \quad k>0.
\] 
Let $\Omega_N^k = Q_1 \setminus \bigcup_{\vec{\ell}} B(x_{\vec{\ell}}, r_N)$ and define the surface measures
\[
\mu_N = \sum_{\vec{\ell}}\mathcal{H}^{d-1}\mres_{\partial B(x_{\vec{\ell}}, r_N)}.
\]
Observe that 
\begin{equation}\label{eq: totalmass}
\mu_N(\R^d)=k^{d-1}\omega_{d-1}
\end{equation}
where $\omega_{d-1}$ is the $\mathcal{L}^{d-1}$-measure of the unit ball in $\R^{d-1}$. We now claim that
\begin{equation}\label{eq:homoconv}
\mu_N\to\mu=k^{d-1}\omega_{d-1}1_{Q_1}\mathcal{L}^d\ \text{in $H^{-1}(\R^d)$.}
\end{equation}
A proof of \eqref{eq:homoconv} is given in the appendix.

We now prove that
\begin{equation}\label{eq:homoconvlambda}
\liminf_{N\to\infty} \lambda_\beta(\Omega_N^k) \ge \beta\, k^{d-1}\, \omega_{d-1}.
\end{equation}
To do this, consider the $L^2$-normalized eigenfuntions $u_N$ of $\Omega_N^k$ (we omit the dependence on $k$ to shorten the notation). We clearly prove \eqref{eq:homoconvlambda} under the nonrestrictive assumption that $\sup_{N} \lambda_\beta(\Omega_N^k)<+\infty$.

\medskip
\textbf{Step 1.} We extend $(u_N)_N$ to a bounded sequence $(\tilde{u}_N)_N\subset H^1(Q_1)$. To do so, we construct a vector field $V_N \in C_c^1(Q_1;\R^d)$, which on each periodicity cell is such that
\[
V_N(x) = x-x_{\vec{\ell}} \text{ on } B(x_{\vec{\ell}},r_N), \quad V_N = 0 \text{ outside }B\left(x_{\vec{\ell}},\frac1{2N}\right) .
\]
This vector field can be constructed to have uniformly bounded Lipschitz constant with respect to $N$: indeed, a radial cut-off function between a ball of radius $k/N^{\frac d{d-1}}$ and the one of radius $1/2N$ has derivative of order $N$, while $|x-x_{\vec{\ell}}|\le \frac{\sqrt d}{2N}$ on each cell.
We now use the Rellich-Pohozaev identity

\begin{align*}
\int_{\Omega_N^k} \Big[ (\operatorname{div}V_N)\Big(\tfrac12|\nabla u_N|^2 - \lambda_\beta u_N)\Big)
- (\nabla u_N \otimes \nabla u_N) : \nabla V_N \Big]\,dx\\
= \int_{\partial\Omega_N^k} \Big[ \partial_\nu u_N\,(V_N\cdot\nabla u_N)
- \Big(\tfrac12|\nabla u_N|^2 - \lambda_\beta u_N\Big)(V_N\cdot\nu)\Big]\,\dH,
\end{align*}
which can be proved testing the equation $-\Delta u_N= \lambda_\beta u_N$ with $V\cdot \nabla u_N$ and integrating by parts, also using the equality
\[
\nabla u_N\cdot((D^2 u_N)\,V) = \frac12\,V\cdot\nabla\big(|\nabla u_N|^2\big).
\]
Recall that, on $\partial B(x_{\vec{\ell}}, r_N)$ one has
\[
V_N(x) = (x-x_{\vec{\ell}}), \qquad
V_N\cdot\nu = (x-x_{\vec{\ell}})\cdot\nu = -|x-x_{\vec{\ell}}| = -\frac k{N^{\frac d{d-1}}}.
\]
Furthermore
\[
V_N\cdot\nabla u_N = (x-x_{\vec{\ell}})\cdot\nabla u_N = -\frac k{N^{\frac d{d-1}}}.\,\partial_\nu u_N.
\]
No term on $\partial Q_1$ appears, as $V_N=0$ on $\partial Q_1$. Using Robin boundary conditions $\partial_\nu u_N=-\beta u_N$, the equality then becomes
\begin{align*}
\int_{\Omega_N^k}
\Big[
(\operatorname{div} V_N)\,|\nabla u_N|^{2}
- 2\,(\nabla u_N \otimes \nabla u_N) : \nabla V_N
- \lambda_\beta\,(\operatorname{div} V_N)\,u_N^{2}
\Big]\,dx
=\\
\sum_{\vec{\ell}\in \{1,\dots, N\}^d}\int_{\partial B(x_{\vec{\ell}},r_N)}
-\frac k{N^{\frac d{d-1}}}\Big[
|\nabla_\tau u_N|^{2}
+\frac {\beta^2-\lambda_\beta}{N^{\frac d{d-1}}}\,u_N^{2}
\Big]\,\dH.
\end{align*}
Now, the eigenfunctions are normalized in $L^2$, and we assume the eigenvalues (and hence the Robin form) to be uniformly bounded; furthermore, $V_N$ are bounded in $C^1$. We therefore get the existence of a constant $C$ depending on $d$, $k$, $\sup_N \lambda_\beta (\Omega^k_N)$, and $\beta$ such that
\[
\sum_{\vec{\ell}\in \{1,\dots, N\}^d}\int_{\partial B(x_{\vec{\ell}},r_N)} |\nabla_\tau u_N|^2 \dH \le {C}{N^{\frac d{d-1}}}=\frac{C}{r_N}.
\]
Define $\tilde u_N$ as the harmonic extension inside each ball; it holds the estimate (see for instance \cite[Page 20]{Foc16}, which also hints at the proof of the inequality)
\[
\int_{B(x_{\vec{\ell}},r_N)} |\nabla\tilde u_N|^2 \dx\le \frac{r_N}{d-1}\int_{\partial B(x_{\vec{\ell}},r_N)} |\nabla_\tau  u_N|^2 \dH\,,
\]
whence
\[
\sum_{\vec{\ell}\in \{1,\dots, N\}^d}\int_{B(x_{\vec{\ell}},r_N)} |\nabla \tilde u_N|^2 \dx \le C.
\]
Combining with the bound on $\|\nabla u_N\|_{L^2(\Omega_N^k)}$ yields that $\|\nabla \tilde u_N\|_{L^2(Q_1)}$ is bounded by a constant depending on $d$, $k$, $\sup_N \lambda_\beta (\Omega^k_N)$, and $\beta$. As we also have
\[
\int_{\partial Q_1}u_N^2 \dH \le \frac{\lambda_\beta (\Omega^k_N)}{\beta}\,,
\]
the claim is proved.

\medskip
\textbf{Step 2.} We now consider the weak-$H^1(Q_1)$ limit $u$ of the sequence $\tilde u_N$, which exists up to a subsequence, and prove the inequality
\begin{equation}\label{eq: lsc}
\liminf_{N\to\infty} \int_{Q_1} u_N^2 \, d\mu_N =\liminf_{N\to\infty} \int_{Q_1} \tilde u_N^2 \, d\mu_N 
\ge k^{d-1}\,\omega_{d-1} \int_{Q_1} u^2 \dx.
\end{equation} 
The first equality is trivial, as $u_N=\tilde u_N$ on the support of $\mu_N$. As for the lower semicontinuity inequality, fix $M>0$ and a smooth cutoff $0\le \phi \le 1$. Then
\[
\liminf_{N\to\infty}  \int_{Q_1} \tilde u_N^2 \, d\mu_N \ge 
\liminf_{N\to\infty} \int_{Q_1} \min\{(\phi \tilde u_N)^2, M\} \, d\mu_N.
\]
The truncations $\min\{(\phi \tilde u_N)^2,M\}$ converge weakly in $H^1_0$ to $\min\{(\phi u)^2,M\}$.
Using \eqref{eq:homoconv} we then have:
\[
\lim_{N\to\infty}\int_{Q_1}  \min\{(\phi \tilde u_N)^2, M\} \, d\mu_N=
\int_{Q_1}  k^{d-1} \omega_{d-1} \min\{(\phi u)^2, M\} \dx.
\]
Passing to the limit $\phi \uparrow 1$ and $M \to \infty$ gives \eqref{eq: lsc}.

\medskip
\textbf{Step 3.} We prove \eqref{eq:homoconvlambda}. To this aim, it suffices to notice that, since the functions $\tilde u_N$ converge strongly to $u$ in $L^2(Q_1)$, so do  then $u_N 1_{\Omega^k_N}$ by Vitali's theorem. As the functions $u_N$ are $L^2$-normalized, it must then hold $\|u\|_{L^2(Q_1)}=1$, so that \eqref{eq:homoconvlambda} directly follows from Step 2.

\medskip
\textbf{Step 4.} We finally prove $M_1\ge 1$. To this, we first consider a sequence $t^k_N\downarrow 1$ so that $|t^k_N \Omega^k_N|=1$ and observe that by the scaling property \eqref{eq: scaling} we deduce
\begin{equation}\label{eq: liminf}
\lambda_\beta(t^k_N \Omega^k_N)\ge \frac{1}{(t^k_N)^2}\lambda_\beta(\Omega^k_N)\,.
\end{equation}
As for the torsion part, we trivially estimate with the test function $w=1$, and together with \eqref{eq: liminf} and \eqref{eq: totalmass} we get
\[
F_\beta(t^k_N \Omega^k_N)
\ge 
\frac{1}{(t^k_N)^{d+1}}
\frac{\lambda_\beta(\Omega^k_N)}
{\beta \left( \mathcal{H}^{d-1}(\partial Q_1) + k^{\,d-1}\omega_{d-1} \right)}
\]
With the liminf as $N\to +\infty$, using \eqref{eq:homoconvlambda}, we have
\[
M_1 \ge \frac{\beta k^{\,d-1}\omega_{d-1}}
{\beta \left( \mathcal{H}^{d-1}(\partial Q_1) + k^{\,d-1}\omega_{d-1} \right)}
\]
which entails the conclusion when $k\to +\infty$.
\end{proof}

\begin{rema}\label{gamma}
An alternative proof of the inequality $M_1\ge1$ can be obtained by means of a $\gamma$-closure argument; we give a short sketch here below.\\
We first observe that, since both $\lambda_\beta$ and $T_\beta$ are continuous with respect to the $\gamma$-convergence, we also have
$$M_1=\sup\big\{F_1(\mu)\ :\ \mu\in\M\big\}$$
where $\M$ denotes the $\gamma$-closure of the class of open sets $\O$ with $|\O|=1$. By approximating a smooth surface $\partial\O$ by means of very rough surfaces, as in Figure \ref{frast}, we obtain the the class $\M$ contains all the problems with a Dirichlet condition on $\partial\O$, hence also the class of $\gamma$-limits of Dirichlet problems. The $\gamma$-closure of Dirichlet problems has been fully identified in \cite{DMM87} as the class of capacitary measures, that is the measures that vanish on all sets of capacity zero. Then we have that $\M$ contains all the capacitary measures.

In fact the class $\M$ is strictly larger than the class of capacitary measures. To see this one could consider the case of smooth domains converging to a domain with a (possibly smooth) fracture (see, for instance, \cite{BG15,BG19,BGT22,CG24}). More precisely, take a smooth surface $S$ in an open set $\O$, an $\eps$ neighborhood $S_\eps$, and $\O_\eps=\O\setminus S_\eps$. We have that the Robin energies on the sets $\O_\eps$
$$\frac12\int_{\O_\eps}|\nabla u|^2\,dx+\beta\int_{\partial\O_\eps}u^2\,d\HH^{d-1}$$
$\Gamma$-converge (as $\eps\to0$) to the energy
$$\frac12\int_{\O\setminus S}|\nabla u|^2\,dx+\beta\int_{\partial\O}u^2\,d\HH^{d-1}+\beta\int_S \left(u_+^2+u_-^2\right)\,d\HH^{d-1},$$
which cannot be represented by a capacitary measure. In the expression above $u_-$ and $u_+$ are the left and right traces of the Sobolev function $u$.

\begin{figure}[h!]
\centering
{\includegraphics[scale=1.2]{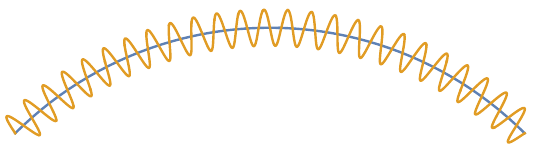}}
\caption{Approximation of a smooth boundary by means of rough boundaries.}
\label{frast}
\end{figure}

In particular, the $\gamma$-closure of Robin problems contains all the capacitary measures $c\,dx$ with $c>0$. Then, taking as $\O$ the ball $B$, we have that
$$M_1\ge\lambda_\beta(B,c)\,T_\beta(B,c)$$
for every constant $c>0$, where
$$\lambda_\beta(B,c)=\min\Bigg\{\frac{\int_B|\nabla u|^2\,dx+\beta\int_{\partial B}u^2\,d\HH^{d-1}+c\int_B u^2\,dx}{\int_B u^2\,dx}\ :\ u\in H^1(B)\setminus\{0\}\Bigg\}$$
and
$$T_\beta(B,c)=\max\Bigg\{\frac{\big(\int_B u\,dx\big)^2}{\int_B|\nabla u|^2\,dx+\beta\int_{\partial B}u^2\,d\HH^{d-1}+c\int_B u^2\,dx}\ :\ u\in H^1(B)\setminus\{0\}\Bigg\}.$$
From the expressions above we obtain
$$\lambda_\beta(B,c)=\lambda_\beta(B)+c$$
and
$$T_\beta(B,c)\ge\frac{1}{\beta\mathcal{H}^{d-1}(\partial B)+c}.$$
Therefore, we have
$$M_1\ge\frac{\lambda_\beta(B)+c}{\beta\mathcal{H}^{d-1}(\partial B)+c}$$
for every $c>0$, and the inequality $M_1\ge1$ now follows by passing to the limit as $c\to+\infty$.
\end{rema}

\section{Open questions and remarks}\label{sopenq}

In this section we raise a couple of open and challenging questions in the case $\beta<+\infty$ of Robin boundary conditions. Because of the lack of scaling property for the functional $F_q$, the proofs available in the Dirichlet case cannot be easily extended, and we believe it would be interesting to prove (or disprove) them.

The first question is concerned with the Kohler-Jobin inequality: in the Dirichlet case it reads as
$$m_q=F_q(B)\qquad\text{for every }q\le2/(d+2).$$
We have seen in Theorem \ref{mqb} that in the Robin case $\beta<+\infty$ the threshold (if any) cannot be the same because $m_q=0$ for every $q>1/(d+1)$. So the first question is:

\begin{itemize}
\item[(Q1)] {\it in the Robin case $\beta<+\infty$, is it true that for every $q\le1/(d+1)$, or at least for $q$ below a smaller threshold, the Kohler-Jobin inequality
$$m_q=F_q(B)$$
holds?}
\end{itemize}

A positive answer to the question (Q1) above would be rather natural since as $q\to0$, the term $T_\beta^q$ in the expression of the functional $F_q$, becomes less and less relevant, and the ball $B$ actually minimizes the limit case $q=0$, where $F_q$ simply reduces to $\lambda_\beta$, as recalled in Section \ref{spreli}.

The second question is concerned with a possible analogous of Theorem \ref{Mqinfty} for $M_q$ when $q$ is large enough (reverse Kohler-Jobin inequality). We have seen that in the Dirichlet case $\beta=+\infty$ we have
$$M_q=F_q(B)$$
for every $q$ above a suitable threshold $\bar q$. So, a natural question is:

\begin{itemize}
\item[(Q2)] {\it in the Robin case $\beta<+\infty$, is it true that for every $q$ above a suitable threshold $\bar{q}_\beta$ the reverse Kohler-Jobin inequality
$$M_q=F_q(B)$$
holds?}
\end{itemize}

Again, a positive answer to the question (Q2) above would be rather natural since as $q\to+\infty$ the term $\lambda_\beta$ in the expression of the functional $F_q$, becomes less and less relevant, and the ball $B$ actually maximizes the limit case $q=+\infty$, where $F_q$ simply reduces to $T_\beta$, as recalled in Section \ref{spreli}.


\bigskip

\noindent{\bf Acknowledgments.} The authors are member of the Gruppo Nazionale per l'Analisi Matematica, la Probabilit\`a e le loro Applicazioni (GNAMPA) of the Istituto Naziona\-le di Alta Matematica (INdAM). S.C. has been partially supported by the project "Stochastic Modeling of Compound Events" Funded by the Italian Ministry of University and Research (MUR) under the PRIN 2022 PNRR (cod. P2022KZJTZ). F.S. has been partially supported by the project "Variational Analysis of Complex Systems in Material Science, Physics and Biology" Funded by the Italian Ministry of University and Research (MUR) under the PRIN 2022 (cod. 2022HKBF5C). PRIN projects are part of PNRR Italia Domani, financed by European Union through NextGenerationEU.

\section*{Appendix}

The following appendix is devoted to the proof of \eqref{eq:homoconv}.

\begin{lemm}[Surface measure homogenization]
The measures $\mu_N$ converge in $H^{-1}(\R^d)$ to
\[
\mu = k^{d-1}\,\omega_{d-1} \,\mathcal{L}^d
\]
where $\omega_{d-1}$ is the area of the unit $(d-1)$-sphere.
\end{lemm}

\begin{proof}
For the sake of brevity, we set 
\[
\nu_N = \mu_N-\mu.
\]
A direct computation shows that $\nu_N$ tends to 0 narrowly; our aim is to show that $\|\nu_N\|_{H^{-1}(\R^d)}\to 0$. By construction it holds $\nu_N(\R^d)=0$. We recall that the $H^{-1}$ norm can be expressed as Fourier integral via
\[
\|\nu_N\|_{H^{-1}(\R^d)}^2=\int_{\R^d}\frac{|\widehat{\nu_N}|^2}{1+|\xi|^2}\;d\xi.
\]
Let us notice that $\widehat{\nu}_N\in C^\infty$ and $\widehat{\nu}_N(0)=0$. In particular, the ratio $\frac{\widehat{\nu}_N(\xi)}{|\xi|}$ is bounded as $\xi\to 0$. Now, denoting by $G$ the Green function of $-\Delta$ in $\R^d$  and recalling that $\widehat{G}=\frac{1}{|\xi|^2}$, we have
\begin{equation}\label{eq: h-1homogen}
\int_{\R^d}\frac{|\widehat{\nu}_N|^2(\xi)}{|\xi|^2}\;d\xi=\langle \widehat{G}\,\widehat{\nu}_N,\overline{\widehat{\nu}}_N\rangle
\end{equation}
where $\langle \cdot, \cdot \rangle$ denotes the  duality product between functions and measures. We thus get the thesis if the last right hand side tends to zero; equivalently, passing to the antitransforms, we prove \footnote{for a rigorous justification of this anti-tranformation, one may first consider mollifiers $\varrho_\varepsilon$ and apply the Parseval identity to the duality between  $G*(\varrho_\varepsilon*\nu_N)$, regarded as a tempered distribution, and the Schwartz function $\varrho_\varepsilon*\nu_N$; additionally,  the formula for the Fourier transform of a convolution is justified in this setting, and, since $\widehat{\varrho_\varepsilon * \nu_N}$ is pointwise dominated by $\widehat{\nu}_N$, one may pass to the limit in the left-hand side of \eqref{eq: h-1homogen}.}  that $\langle G*\nu_N,\nu_N\rangle\to 0$.
This latter limit can be splitted as follows:
$$\langle G*\mu_N,\mu_N\rangle-2\langle G*\mu,\mu_N\rangle+\langle G*\mu,\mu\rangle.$$
As $G*\mu$ is a $C^1$ function by elliptic regularity, and using the narrow convergence of $\mu_N$ to $\mu$ (and their compact support), it suffices to show that
\begin{equation}\label{eq: aim}
\lim_{N\to +\infty}\langle G*\mu_N,\mu_N\rangle=\langle G*\mu,\mu\rangle\,.
\end{equation}
To this aim, let us define, for any $\varepsilon>0$, the continuous functions
$$
G_\varepsilon(x):=\left[G(x)-G(\varepsilon)\right]1_{\left\{|x|\ge\varepsilon\right\}}.
$$
where, with a small abuse of notation, we denoted $G(|x|)=G(x)$ since $G$ is radial. For any fixed $\varepsilon>0$ it holds
\begin{equation}\label{eq: conv1}
\lim_{N\to +\infty}\left|\langle G_\varepsilon*\mu_N,\mu_N\rangle-\langle G_\varepsilon*\mu,\mu\rangle\right|= 0
\end{equation}
since the sequence $(G_\varepsilon*\mu_N)_N$ converges to $G_\varepsilon*\mu$ uniformly. Moreover, by monotone convergence, one gets
\begin{equation}\label{eq: conv2}
\lim_{\varepsilon\to 0}\left|\langle G_\varepsilon*\mu,\mu\rangle-\langle G*\mu,\mu\rangle\right|=0
\end{equation}
It further holds
\begin{align*}
\left|\langle G_\varepsilon*\mu_N,\mu_N\rangle-\langle G*\mu_N,\mu_N\rangle\right|\le & 2\int_{\left\{|x-y|\le\varepsilon\right\}}G(x-y)\, d\mu_N(x)\, d\mu_N(y)\\
=2\sum_{\vec{\ell}}\int_{\left(\partial B(x_{\vec{\ell}}, r_N)\times\partial B(x_{\vec{\ell}}, r_N)\right)\cap\left\{|x-y|\le\varepsilon\right\}}& G(x-y) \dH(x)\dH(y)\\
+2\sum_{\vec{\ell}\neq \vec{p}} \int_{\left(\partial B(x_{\vec{\ell}}, r_N)\times\partial B(x_{\vec{p}}, r_N)\right)\cap\left\{|x-y|\le\varepsilon\right\}}& G(x-y) \dH(x)\dH(y)
\end{align*}

Let us split the analysis into off-diagonals and diagonal terms. 
We only discuss here the case $d\ge 3$. Concerning the off diagonal terms, we observe that, for any $i\neq j$
$$|x-y|=|(x-x_i)-(y-x_j)+(x_{\vec{\ell}}-x_{\vec{p}})|\ge|x_{\vec{\ell}}-x_{\vec{p}}|-2r_N\ge\frac{1}{N}-\frac{k}{N^{\frac d{d-1}}},$$
thus, for $N$ sufficiently large,
$$2\varepsilon\ge 2|x-y|\ge|x_{\vec{\ell}}-x_{\vec{p}}|\ge\frac{1}{N}.$$
We thus get
\begin{align*}
\sum_{\vec{\ell}\neq \vec{p}} &\int_{\partial B(x_{\vec{\ell}}, r_N)} \int_{\partial B(x_{\vec{p}}, r_N)} G(x-y) 1_{\left\{|x-y|\le\varepsilon\right\}}\dH(x)\dH(y)\\
&\le 2^{d-2} \sum_{\vec{\ell}\neq \vec{p}} 1_{\left\{|x_{\vec{\bar{\ell}}}-x_{\vec{p}}|\le2\varepsilon\right\}}\int_{\partial B(x_{\vec{\ell}}, r_N)} \int_{\partial B(x_{\vec{p}}, r_N)} \frac{1}{|x_{\vec{\ell}}-x_{\vec{p}}|^{d-2}}\dH(x)\dH(y)\\
&\le\frac{2^{d-2}k^{d-1}}{N^{2d}} \sum_{\vec{\ell}\neq \vec{p}}\frac{1_{\left\{|x_{\vec{\ell}}-x_{\vec{p}}|\le2\varepsilon\right\}}}{|x_{\vec{\ell}}-x_{\vec{p}}|^{d-2}}=\frac{C}{N^d}\sum_{\vec{\ell}}\left(\frac{1}{N^d}\sum_{\vec{p}}\frac{1_{\left\{|x_{\vec{\ell}}-x_{\vec{p}}|\le2\varepsilon\right\}}}{|x_{\vec{\ell}}-x_{\vec{p}}|^{d-2}}\right)\\
&{\le}\frac{C}{N^d}\sum_{\vec{p}}\frac{1_{\left\{|x_{\vec{\bar{\ell}}}-x_{\vec{p}}|\le2\varepsilon\right\}}}{|x_{\vec{\bar{\ell}}}-x_{\vec{p}}|^{d-2}}.
\end{align*}
for an arbitrary choice of ${\vec{\bar{\ell}}}\in \{1,\dots N\}^d$, since the second sum gives by simmetry a result which is independent of ${\vec{\bar{\ell}}}$.  Now, the right-hand side above can be seen as a Riemann sum as $N\to +\infty$ for the integrable function $\frac1{|x_{\vec{\bar{\ell}}}-x|^{d-2}}$, so that we get
$$
\frac{1}{N^d}\sum_{\vec{p}}\frac{1_{\left\{|x_{\vec{\bar{\ell}}}-x_{\vec{p}}|\le2\varepsilon\right\}}}{|x_{\vec{\bar{\ell}}}-x_{\vec{p}}|^{d-2}}\to\int_{\left\{|x-x_{\vec{\bar{\ell}}}|\le 2\varepsilon\right\}}\frac{1}{|x-x_{\vec{\bar{\ell}}}|^{d-2}}\,dx=C\varepsilon^2.
$$
On the other hand, when $N$ is large enough, with the change of variables $x=x_{\vec{\bar{\ell}}}+r_N u$, $y=x_{\vec{\bar{\ell}}}+r_N v$ the diagonal terms can be estimated with
$$\frac{1}{N^d}\int_{\mathbf{S}^{d-1}\times\mathbf{S}^{d-1}}\frac{1}{|u-v|^{d-2}}\dH(u)\dH(v)\to 0,$$
as the integral is finite thanks, for instance, to \cite[Theorem 9.7]{LL01}.

Putting together the above estimates with \eqref{eq: conv1} and \eqref{eq: conv2}, we get \eqref{eq: aim} by first letting $N \to +\infty$ and then $\varepsilon \to 0$, and this concludes the proof.
\end{proof}

\bigskip\bigskip\small\noindent
Giuseppe Buttazzo: Dipartimento di Matematica, Universit\`a di Pisa\\
Largo B. Pontecorvo 5, 56127 Pisa - ITALY\\
{\tt giuseppe.buttazzo@unipi.it}\\
{\tt http://www.dm.unipi.it/pages/buttazzo/}

\bigskip\small\noindent
Simone Cito: Dipartimento di Matematica e Fisica, Universit\`a del Salento\\
Via Arnesano, 73100 Lecce - ITALY\\
{\tt simone.cito@unisalento.it}\\
{\tt https://www.unisalento.it/people/simone.cito}

\bigskip\small\noindent
Francesco Solombrino: Dipartimento di Scienze e Tecnologie Biologiche ed Ambientali, Universit\`a del Salento\\
Via Arnesano, 73100 Lecce - ITALY\\
{\tt francesco.solombrino@unisalento.it}\\
{\tt https://www.unisalento.it/people/francesco.solombrino}

\end{document}